\documentclass{article}
\usepackage{amsmath,amsthm,amssymb}
\usepackage[all]{xy}

\newtheorem{theor}{Theorem}
\newtheorem{propo}{Proposition}
\newtheorem{lem}{Lemma}
\newtheorem{coro}{Corollary}
\theoremstyle{definition}
\newtheorem{defi}{Definition}
\newtheorem{rema}{Remark}

\newcommand{\Z}{\mathbb Z}
\newcommand{\V}{{\mathcal V}}
\newcommand{\E}{{\mathcal E}}
\newcommand{\eps}{\epsilon}
\newcommand{\wt}{\widetilde}
\newcommand{\ot}{\leftarrow}
\newcommand{\ol}{\overline}
\newcommand{\cor}{\mathrm{corank} }
\newcommand{\rank}{\mathrm{rank} }
\newcommand{\sgn}{\mathrm{sgn}}
\newcommand{\Kh}{\mathrm{\overline{Kh}}}
\newcommand{\ofo}{\Omega^{PU}_2}

\newcommand{\tar}[2]{\begin{array}{c} #1 \\ \mbox{#2}\end{array}}
\newcommand{\bega}{\left(\begin{array}}
\newcommand{\ena}{\end{array}\right)}

\title{Odd Khovanov homology of principally unimodular bipartite graph-links}
\author{Igor Nikonov}
\date{}

\begin{document}
\maketitle

\begin{abstract}
We define odd Khovanov homology over $\Z$ for principally unimodular
bipartite graph-links.
\end{abstract}

\section{Introduction}

This article is a sequel of our paper~\cite{N} and describes the
integral version of odd Khovanov homology for graph-links. We refer
the reader to papers~\cite{IM,IM1} for the definition of graph-links
and to papers~\cite{Blo,ORS} for the construction of odd Khovanov
homology.

The definition of integer-valued odd Khovanov homology for
graph-links faces the difficulty that the signs of the integer
intersection matrix of an oriented chord diagram does not survive in
general after mutations. But we can keep the signs if one is able to
distinguish outer and inner chords. This reason confines us to
bipartite graph-links. There is another restriction, which is
necessary to retain the orientation after Reidemeister moves, ---
{\em principal unimodularity}~\cite{BCG}. So the definition of
integer odd Khovanov homology is given for bipartite principally
unimodular graph-links.

Bipartite principally unimodular graph-links can be considered as
"classical" graph-links because the realizable graphs of this type
are intersection graphs of chord diagrams that correspond to
classical links. The question is whether this class of graph-links
gives the "natural" definition of classical graph-links. Another
problem is if the theory is meaningful, i.e. does there exist
non-realizable bipartite principally unimodular graph-link.

\section{Principally unimodular bipartite graph-links}

In this section we define an orientable version of the Reidemeister
on oriented bipartite graphs.

Let G be an oriented bipartite graph without loops and multiple
edges and $\V=\V(G)$ be the set of its vertices. We assume $G$ be a
{\em labeled} graph, i.e. every vertex in $G$ is endowed with sign
'+' or '-'. In other words, there is a map $\sgn : \V\to \{-1,1\}$.

Fix an enumeration of vertices for $G$. We define the {\em adjacency
matrix} $A(G)=(a_{ij})_{i,j=1,\dots,n}$ over $\Z_2$ as follows:
$a_{ij}=1$ and $a_{ji}=-1$ and if and only if $v_i$ it the beginning
and $v_j$ is the end of an edge in the graph $G$ (we shall denote
this situation as $v_i\to v_j$) , and $a_{ij}=0$ $v_i$ and $v_j$ are
not adjacent. Besides we set $a_{ii}=0$.

Any subset $s\subset \V$ we shall call a {\em state}. Let's define
$G(s)$ to be the complete subgraph in $G$ with the set of vertices
$s$ and denote $A(s)=A(G(s))$. Since $G$ is bipartite the set of
vertices splits into a disjoint sum $\V=\V_0\sqcup\V_1$ and for
every state $s$

\begin{equation}
A(s) =\bega{cc} 0& B(s)\\ -B(s)^\top& 0\ena
\end{equation}
where the rows of the matrix $B(s)$ correspond to the vertices in
$s\cap\V_0$ and the columns correspond to the vertices in
$s\cap\V_1$.

Let $v\in \V$. The set of all vertices in $\V$ adjacent to $v$ is
called {\em neighbourhood} of the vertex $v$ and denoted $N(v)$.

Let us define the Reidemeister moves for labeled oriented bipartite
graphs.

$R$. For a given vertex $v\in\V(G)$ we change the direction of all
the edges incident to $v$.

$\Omega_1$. The first Reidemeister move is an addition/removal of an
isolated vertex labeled $+$ or $-$.

$\Omega_2$. The second Reidemeister move is an addition/removal of
two nonadjacent vertices  $u$ and $v$ having the different signs and
the same neighbourhoods so that the new graph remain bipartite. We
require the orientations of the new edges to be compatible: for any
vertex $w\in\V(G)$ we have $u\to w$ (resp. $u\ot w$) if and only if
$v\to w$ (resp. $v\ot w$).

$\Omega_3$. The third Reidemeister move is defined as follows. Let
$u, v, w$ be three vertices of $G$ with signs '-' and $u$ be
adjacent only to $v$ and $w$ so that $u\to v$ and $u\to w$. Then we
disconnect $u$ from $v$ and $w$. We set $u\to t$ (resp. $u\ot t$)
for all $t$ such that $v\to t$ (resp. $v\ot t$) and  set $u\to t$
(resp. $u\ot t$) if $w\ot t$ (resp. $w\to t$). In addition, we
change the labels of $v$ and $w$ to '+'. The inverse operation is
also called the third Reidemeister move.

$\Omega_4$. The fourth Reidemeister move is defined as follows. We
take two adjacent vertices $u$ labeled $a$ and $v$ labeled $b$. Then
we change the label of $u$ to $-b$ and the label of  $v$ to $-a$ and
change also the orientation of the edge $uv$. After that we change
the adjacency for each pair $(t,w)$ of vertices where $t\in N(u)$
and $w\in N(v)$. We set the orientation of a new edge $tw$ so that
the square $utwv$ be {\em even}, i.e. the number of codirectional
edges in the round $utwv$ be even (see examples below).
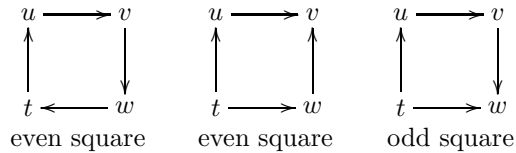
\begin{gather*}
\tar{\xymatrix{ u\ar[r] & v\ar[d]\\
t\ar[u] & w\ar[l]}}{even square}\quad
\tar{\xymatrix{ u\ar[r] & v\\
t\ar[r]\ar[u] & w\ar[u]}}{even square}\quad
\tar{\xymatrix{ u\ar[r] & v\ar[d]\\
t\ar[r]\ar[u] & w}}{odd square}
\end{gather*}

\begin{propo}
Let $G$ be an oriented bipartite labeled graph and $\wt G$ differ
from $G$ by orientation of edges. Then we can obtain $\wt G$ by
applying the moves $\Omega_2$ and $\Omega_4$ to the graph $G$.
\end{propo}
\begin{proof}
Let $u$ and $v$ be two adjacent vertices in $G$. We can change the
direction of the edge connecting $u$ and $v$ in the following way.
We add two vertices $w, w'$ such that $N(w)=N(w')=\{v\}$ ($\Omega_2$
move). Then we add two vertices $t, t'$ such that $N(t)=N(t')=\{u,
w\}$ and the square $uvwt$ is odd (another $\Omega_2$ move). Denote
the obtained graph $G'$. Then we apply twice $\Omega_4$ move to the
pair of vertices $w, t$. In the new graph $G''$ the adjacency of
vertices are the same as in $G'$ and the directions of edges remain
unchanged except the edge $uv$ which changes the direction because
the square $uvwt$ in $G''$ is even. Then we remove vertices $w, w',
t, t'$ to obtain the graph $G'''$ which differs from $G$ by the
direction of the edge $uv$.

Repeating this operation we get to the graph $\wt G$.
\end{proof}

This statement shows that the theory of oriented bipartite graphs
with moves $R, \Omega_1,\dots,\Omega_4$ is in fact the theory of
{\em undirected} labeled bipartite graphs with the usual
Reidemeister moves of graph-links which preserve the bipartite
structure of the graphs. So we have to impose some additional
constraints to make orientation of graphs significant.

\begin{defi}
Let $G$ be an oriented bipartite labeled graph. We call the
orientation of $G$ {\em principally unimodular} if for each state
$s\subset\V$ we have $\det A(s)$ is equal either $0$ or $1$. The
graph $G$ we call {\em PU-oriented}.
\end{defi}

Any bipartite graph, which is realizable as the intersection graph
of a chord diagram, is PU-oriented~\cite{B}. The inverse statement
is not true, this graph has a PU-orientation but is not realizable:
$$
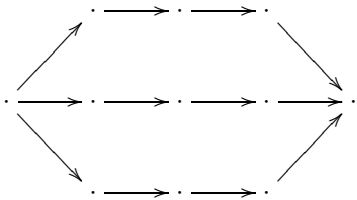
\xymatrix{ &\cdot\ar[r]&\cdot\ar[r]&\cdot\ar[dr]&\\
\cdot\ar[ur]\ar[r]\ar[dr]&\cdot\ar[r]&\cdot\ar[r]&\cdot\ar[r]&\cdot\\
&\cdot\ar[r]&\cdot\ar[r]&\cdot\ar[ur]&}
$$
The question whether exists a bypartite PU-oriented graph that can
not be transformed by Reidemeister moves into a realizable graph is
still open.

\begin{lem}\label{crit1}
Let $G$ be an oriented bipartite labeled graph. These statements are
equivalent:
\begin{enumerate}
\item $G$ is PU-oriented;
\item any minor of the matrix $A(G)$ is equal to $0, -1$ or $1$;
\item any minor of the matrix $B(G)$ is equal to $0, -1$ or
$1$.
\end{enumerate}
\end{lem}
\begin{proof}
1 $\Rightarrow$ 3. Let $B$ be a square submatrix of $B(G)$. Denote
$s_0\in\V_0$ the set of vertices that correspond the rows of $B$ and
$s_1\in\V_1$ the set of vertices that correspond the columns of $B$.
Then we have
$$ A(s_0\cup s_1)=\bega{cc} 0 & B\\ -B^\top &0\ena $$
so $\det A(s_0\cup s_1) =(\det B)^2=0$ or $1$. Then $\det B=0$, $1$
or $-1$.

3 $\Rightarrow$ 2. Let $C$ be a square matrix in $A(G)$. According
to the splitting $\V=\V_0\sqcup\V_1$ the matrix $C$ can be written
in the block form
$$
C= \bega{cc} 0 & B_1\\ -B_2^\top &0\ena.
$$
If the blocks $B_1$ and $B_2$ are not square then the rank of $C$ is
less then the size of $C$ so $\det C=0$. Otherwise, $\det C=\pm\det
B_1\det B_2=0, -1$ or $1$ because $\det B_1,\det B_2=0,-1$ of $1$.

The implication 2 $\Rightarrow$ 1 is obvious.
\end{proof}

\begin{propo}\label{faith_invar}
Let $G$ be PU-oriented. Then
\begin{enumerate}
\item any subgraph of $G$ is PU-oriented;
\item if $G'$ is obtained from $G$ by applying the moves
$R,\Omega_1,\Omega_3,\Omega_4$ then $G'$ is PU-oriented.
\end{enumerate}
\end{propo}

\begin{proof}
The first statement of the proposition is evident.

Invariance under $R$ move. The matrix $A(G')$ is obtained from
$A(G)$ by multiplication by $-1$ the row and the column
corresponding to the vertex $v$ of the move $R$. Then for any state
$s\in\V$ we have $\det A(G'(s))=\det A(G(s))$ if $v\not\in s$ and
$\det A(G'(s))=(-1)^2\det A(G(s))=\det A(G(s))$ if $v\in s$. So all
the determinants are equal to $0,-1$ or $1$, hence $G'$ is
PU-oriented.

Invariance under $\Omega_1$ move. Assume we obtain $G'$ by adding an
isolated vertex $v$. Then for any state $s\in\V(G')$ we have $\det
A(G'(s))=0$ if $v\in s$ and $\det A(G'(s))=\det A(G(s))$ if
$v\not\in s$. Thus $G'$ is PU-oriented.

Invariance under $\Omega_3$ move. Without loss of generality we can
renumber $\V(G)=\V(G')$ so that the vertices $u, v, w$ of the move
have the indices $1,2,3$. Then the adjacency matrices have the form
$$
A(G)=\bega{cccc} 0&1&1&0\\
-1&0&0&a\\
-1&0&0&b\\
0^\top&-a^\top&-b^\top&C\ena,\
A(\wt G)=\bega{cccc} 0&0&0&a-b\\
0&0&0&a\\
0&0&0&b\\
(b-a)^\top&-a^\top&-b^\top&C\ena.
$$

We only need to check that the vector $a-b$ is the adjacency vector
of the vertex $u$ in $G'$. Let us consider a vertex $t$ in $G$ and
let $p$ be its index. Denote $A(G)=(a_{ij})_{i,j=1,\dots, n}$
$A(G')=(a'_{ij})_{i,j=1,\dots, n}$. If $t$ is not adjacent to $v$
and $w$ then $a_{2p}=a_{3p}=0$ and $a'_{1p}=0$ so
$a'_{1p}=a_{2p}-a_{3p}$. If $t$ is adjacent to $v$ and not adjacent
to $w$ then $a_{3p}=0$ and $a'_{1p}=a_{2p}=a_{2p}-a_{3p}$. If $t$ is
adjacent to $w$ and not adjacent to $v$ then $a_{2p}=0$ and
$a'_{1p}=-a_{3p}=a_{2p}-a_{3p}$. If $t$ is adjacent to $v$ and $w$
then $a_{2p}=a_{3p}$ and $a'_{1p}=0=a_{2p}-a_{3p}$. The case
$a_{2p}=1,\ a_{3p}=-1$ (or $a_{2p}=-1,\ a_{3p}=1$) is impossible
because we would have
$$
A(G(\{u,v,w,t\}))=\bega{cccc} 0& 1&1&0\\
-1& 0&0&1\\
-1& 0&0&-1\\
0& -1&1&0\ena
$$
with $\det A(G(\{u,v,w,t\}))=4$ so $G$ would not be PU-orientable.

For any state $s\in\V\setminus \{u,v,w\}$ we have $\det A(G(s))=\det
A(G'(s))$, $\det A(G(s\cup\{v\}))=\det A(G'(s\cup\{v\}))$, $\det
A(G(s\cup\{w\}))=\det A(G'(s\cup\{w\}))$, $\det
A(G(s\cup\{v,w\}))=\det A(G'(s\cup\{v,w\}))$. There are also
equalities $\det A(G(s\cup\{u\}))=0$, $\det A(G(s\cup\{u,v\}))=\det
A(G(s\cup\{u,w\}))=\det A(G(s))$, $\det A(G(s\cup\{v,w\}))=\det
A(G'(s\cup\{u,v\}))=A(G'(s\cup\{u,w\}))$, $\det
A(G(s\cup\{u,v,w\}))=\det A(G'(s\cup\{u\}))$ and $\det
A(G'(s\cup\{u,v,w\}))=0$.

We see that all the determinants $\det A(G(s)),\ s\in\V$ are equal
to some determinants of the graph $G'$ and vice versa. Thus, $G$ is
PU-oriented iff $G'$ is PU-oriented.

Invariance under $\Omega_4$ move. Without loss of generality we can
suppose that $u\in\V_0$, $v\in\V_1$ and $u\rightarrow v$ where $u,v$
are the vertices of the move $\Omega_4$. We can also assume that $u$
is the first vertex in $\V_0$ and $v$ is the first in $\V_1$. Then
the matrix $B(G)$ looks like
$$
B(G)=\bega{cc} 1 & c\\
d & B_1\ena.
$$
Adding/subtracting the first row of $B(G)$ from others we obtain the
matrix
$$
\wt B = \bega{cc} 1 & c\\
0 & \wt B_1\ena.$$ The elements of $\wt B_1$ are $0,\pm 1$ because
otherwise we would have (up to sign change of rows and columns) a minor $\left|\begin{array}{cc} 1 & 1\\
-1& 1\end{array}\right|=2$ in $B(G)$. Let $D_1$ be any square matrix
in $\wt B_1$ and $D$ be the matrix in $\wt B$ obtained from $D_1$ by
adding the first row and the first column. Then $\det D_1=\det
D=0,\pm 1$ because $D$ is equivalent by row transformation to a
square submatrix in $B(G)$. In particular, $\det \wt B_1=\det B$.
The matrix $\wt B_1$ can be also obtained by adding/subtraction the
column $d$ to the columns of $B_1$.

The matrix $B(G')$ is equal to $
\bega{cc} -1 & c\\
d & \wt B_1\ena$. Let us consider any square submatrix $D$ in
$B(G')$. If $D$ contains neither the first row nor the first column
then $D$ is a submatrix of $\wt B_1$ so $\det D= 0,\pm1$. If $D$
contains the first row but not the first column it is equivalent (by
row transformations) to a submatrix in $B(G)$. If $D$ contains the
first column but not the first row it is equivalent (by column
transformations) to a submatrix in $B(G)$. If $D=\bega{cc} -1 & c'\\
d' & C\ena$ then it can be transformed (by operations on rows) to
the matrix $\bega{cc} -1 & c'\\
0 & \wt C\ena$ where $\wt C$ coincides with a submatrix in $B_1$.
Hence $\det D=-\det\wt C=0,\pm 1$.

Thus, any minor of the matrix $B(G')$ is equal to $0,\pm1$ and by
Lemma~\ref{crit1} the graph $G'$ is PU-oriented.
\end{proof}

The set of PU-oriented graphs is not stable under the second
Reidemeister move. So we define the {\em principally unimodular
second Reidemeister move} $\ofo$ by requiring the result to be a
PU-oriented bipartite graph.

\begin{defi}
A {\em PU-oriented graph-link} is the class of equivalence of a
PU-oriented bipartite labeled graph modulo moves $R,\Omega_1,
\ofo,\Omega_3, \Omega_4$.
\end{defi}

Let us consider several properties of PU-oriented graph.


\begin{propo}\label{stable_even}
Let $G$ be an oriented bipartite labeled graph. Then $G$ is
PU-oriented if and only if any graph $G'$ obtained from $G$ by a
sequence of moves $\Omega_4$ does not contain odd squares
(4-cycles).
\end{propo}
\begin{proof}
Necessity of the condition follows from
Proposition~\ref{faith_invar}.

Assume that applying $\Omega_4$ to $G$ doesn't generate odd squares.
We shall call such graphs stably even. Let us consider a state
$s\in\V$ and $s_i=s\cap V_i,\ i=0,1$. If $\#s_0\ne \#s_1$ then $\det
A(s)=0$. So assume $\#s_0\ne \#s_1=k$. We shall prove that $\det
A(s)=(\det B(s))^2=0$ or $1$ by induction on $k$.

For $k=1$ the statement is obvious. If $k=2$ then $\det B(s)=0,\pm1$
because $G$ has no odd squares.

Assume that for any stably even graph $G'$ and $s'_i\in V_i(G'),\
i=1,2$ such that $\#s'_0=\#s'_1<k$ we have $\det B(G'(s'_0\cup
s'_1))=0,\pm 1$.

If the matrix $B(s)$ contains only zeros then $\det B(s)=0$.
Otherwise without loss of generality we can suppose that
$B(s)=\bega{cc} 1 & c\\
d & B_1\ena.$ Then apply the move $\Omega_4$ to the first vertex $u$
in $s_0$ and the first vertex $v$ in $s_1$. Denote the obtained
graph as $G'$. Denote also $s'=s\setminus\{u,v\}$. Then we have $ B(G'(s))=\bega{cc} -1 & c\\
d & \wt B_1\ena$ where $\det \wt B_1=\det B(s)$ (see the proof of
$\Omega_4$-invariance in Proposition~\ref{faith_invar}). But $\wt
B_1= B(G'(s'))$, $G'$ is stably even and
$\#s'_0=\#s'_1<\#s_0=\#s_1=k$. Therefore, $\det B(s)=\det \wt
B_1=0,\pm1$.

Thus, for any state $s$ of $G$ we have $\det B(s)=0,\pm1$ so $G$ is
PU-oriented by Lemma~\ref{crit1}.
\end{proof}

\begin{defi}
We call a cycle $C$ in an oriented graph {\em chordless} if
$C=G(\V(C))$, that is if any two vertices in $C$, which are not
neighbours in $C$, are not adjacent in $G$. A cycle $C$ is {\em
even} if the number of codirectional edges of the cycle is even.
Otherwise the cycle is {\em odd}.
\end{defi}

\begin{propo}\label{primitive_even}
Let $G$ be a PU-oriented bipartite graph. Then any chordless cycle
in $G$ is even.
\end{propo}
\begin{proof}
Let us consider a chordless cycle $C$. If the length $l$ of $C$ is
$4$ then $C$ is a square and, thus, is even. If $l>4$ we choose any
two neighbour vertices $u$ and $v$ of $C$ and apply the move
$\Omega_4$. In the new graph the cycle $C$ splints into an even
square that contains $u$ and $v$ and a chordless cycle $C'$ of
length $l-2$. The new cycle has the same parity as $C$. After
repeating this procedure $l/2-2$ times, the rest of the cycle $C$
will be a square which is even due to Proposition~\ref{stable_even}.
Therefore the original cycle $C$ is even.
\end{proof}

\begin{coro}
Any two PU-orientations of a bipartite graph $G$ coincide up to
reversions $R$.
\end{coro}
\begin{proof}
The difference between two principally unimodular orientations
defines a cocycle $c$ in $H^1(G,\Z_2)$. By
proposition~\ref{primitive_even} the cocycle $c$ vanishes on any
chordless cycle. But chordless cycles generate $H_1(G,\Z_2)$,
therefore $c=0$. Then $c=d\alpha,\ \alpha\in C^0(G,\Z_2)$. This
means that one of the considered PU-orientations can be obtained
from the other by $R$ moves at the vertices $v$ such that
$\alpha(v)\ne 0$.
\end{proof}

Thus the theory of PU-oriented graph-links is in fact a theory of
{\em PU-orientable} graph-links. We shall call graphs, which admit
PU-orientation, {\em principally unimodular graphs} (or {\em
PU-graphs}) and call the corresponding graph-links {\em
PU-graphs-links}.

\section{Odd Khovanov homology of PU-graph-links}

Let $G$ be a PU-oriented bipartite labeled graph with $n$ vertices
and $A=A(G)$ be its adjacency matrix.

Suppose $s\subset \V=\V(G)$.  Consider the vector space $$ V(s) =
\Z<x_1, \dots, x_n\,|\, r^s_1,\dots, r^s_n>$$ where relations
$r^s_1,\dots, r^s_n$ are given by the formula
\begin{equation}
r^s_i=\left\{\begin{array}{cl} x_i-\sum\limits_{\{j\,|\,v_j\in s\}}
\sgn(v_j)a_{ij}x_j,& \mbox{if }v_i\not\in s,\\
-\sum\limits_{\{j\,|\,v_j\in s\}} \sgn(v_j)a_{ij}x_j,& \mbox{if
}v_i\in s
 \end{array}\right.
\end{equation}

We denote $\ol A = (\ol a_{ij})_{i,j=1,\dots,n},\ \ol
a_{ij}=-\sgn(v_j)a_{ij},$ the relation matrix. Let $A'$ be a
submatrix in $A(G)$. We denote $\ol A'$ the corresponding submatrix
(i.e. submatrix with the same sets of rows and columns as $A'$) in
$\ol A$. We have $\rank A'=\rank \ol A'$ and $\rank A'=\rank \ol
A'$.

The main technical consequence of the principal unimodularity is the
following statement.

\begin{propo}
For any state $s$ the $\Z$-module $V(s)$ is free.
\end{propo}
\begin{proof}
The module $V(s)$ has no torsion if and only if for any $k$ the
ideal $E_k(s)\subset \Z$ generated by all minors of corank $k$ in
the relation matrix $\ol A(s)$ is equal to $0$ or $\Z$. But by
Lemma~\ref{crit1} all minors in $A$ are equal $0$ or $\pm1$, hence
every minor in $\ol A$ is equal $0$ or $\pm1$.
\end{proof}

The rank of $V(s)$ is equal to $\cor\ol A(s)=\cor A(s)$.

There is a natural bijection between states $s\subset\V$ and
vertices of the hypercube $\{0,1\}^n$. Every edge of the hypercube
is of the type $s\rightarrow s\oplus i$ where $s\oplus i$ denotes
$s\cup\{v_i\}$ if $v_i\not\in s$ and $s\setminus\{v_i\}$ if $v_i\in
s$. We orient the arrow so that $v_i\not\in s$ if $\sgn(v_i)=-1$ and
$v_i\in s$ if $\sgn(v_i)=1$.

We assign to every edge $s\rightarrow s\oplus i$ the map
$\partial_{s\oplus i}^s : \bigwedge^*V(s)\to\bigwedge^*V(s\oplus i)$
of exterior algebras defined by the formula
\begin{equation}
\partial_{s\oplus i}^s(u)=\left\{\begin{array}{cl} x_i\wedge u& \mbox{if }x_i=0\in V(s),\\
u& \mbox{if }x_i\ne 0\in V(s)
 \end{array}\right.
\end{equation}

\begin{lem}
$x_i=0\in V(s)$ iff $\cor A(s\oplus i) =\cor A(s)+1$
\end{lem}
\begin{proof}
Case 1. $\sgn(v_i)=-1$. Then $v_i\not\in s$ and $s\oplus i =
s\cup\{v_i\}$. The relation matrix of $s\oplus i$ up to numeration
of vertices looks like
$$ \ol A(s\oplus i) = \left(\begin{array}{cc}
\ol A(s)& -\ol {a^\top}\\
\ol a & 0
\end{array}\right).$$
and we have $x_i=\sum\limits_{\{j\,|\,v_j\in s\}} \ol a_{ij}x_j$.
Equality $x_i=0$ means that row $\ol a$ is linearly dependent on
rows of the matrix $\ol A(s)$. This is equivalent to the equality
$\rank(\ol A(s)) = \rank \left(\begin{array}{c}
\ol A(s)\\
\ol a
\end{array}
\right)$. So $\rank \ol A(s)\le \rank\ol A(s\oplus i)\le \rank\ol
A(s)+1$. But the ranks of $\ol A(s)$ and $\ol A(s\oplus i)$ are even
because $\rank A(s) = \rank \ol A(s)$, $\rank A(s\oplus i) = \rank
\ol A(s\oplus i)$ and the matrices  $A(s)$ and $A(s\oplus i)$ are
skew-symmetric. Then $\rank A(s\oplus i) = \rank A(s)$ and $\rank
A(s\oplus i)  = \rank A(s)+1$

Case 2. $\sgn(v_i)=1$. Then $v_i\in s$ and $s\oplus i =
s\setminus\{v_i\}$. The intersection matrix of $s$ up to numeration
of vertices has the form
$$ \ol A(s) = \left(\begin{array}{cc}
\ol A(s\oplus i)& -\ol{a^\top}\\
\ol a & 0
\end{array}\right).$$
Since $v_i\in S$, equality $x_i=0$ means that the ranks of the
matrices $ \left(\begin{array}{cc}
\ol A(s\oplus i)& -\ol{a^\top}\\
a & 0
\end{array}\right)$ and $ \left(\begin{array}{cc}
\ol A(s\oplus i)& -\ol{a^\top}\\
\ol a & 0\\
0& 1
\end{array}\right)$ coincide. But $$ \rank\left(\begin{array}{cc}
\ol A(s\oplus i)& -\ol{a^\top}\\
\ol a & 0\\
0& 1
\end{array}\right)=\rank\left(\begin{array}{cc}
\ol A(s\oplus i)& 0\\
0 & 1\\
\ol a& 0
\end{array}\right)\ge\rank \ol A(s\oplus i) +1.$$ So $\rank A(s) =\rank A(s\oplus i) +2$ and
$\cor A(s\oplus i) = \cor A(s)+1$.

One can see that the reasoning can be reverted and the corank
condition is equivalent to the equality $x_i=0$.
\end{proof}

\begin{coro}
$x_i=0\in V(s)$ iff $x_i\ne 0\in V(s\oplus i)$. \hfill$\square$
\end{coro}

\begin{propo}[Correctness of chain maps]
For any state $s$ and index $i$ the map $\partial_{s\oplus i}^s :
\bigwedge^*V(s)\to \bigwedge^*V(s\oplus i)$ is well defined.
\end{propo}
\begin{proof}
We must check that for any element $u$ and any index $j$ there exist
elements $u_k\in V(s\oplus i)$ such that
$$\partial_{s\oplus i}^s(r^s_j\wedge u) = \sum_k
r^{s\oplus i}_k\wedge u_k\ \in V(s\oplus i).$$

For any $j$ we have $r_j^s = r_j^{s\oplus i} +\alpha_j x_i$ for some
$\alpha_j$. If $x_i=0\in V(s\oplus i)$ then
$$\partial_{s\oplus
i}^s(r^s_j\wedge u)=r^s_j\wedge u = r^{s\oplus i}_j\wedge u+\alpha_j
x_i\wedge u=r^{s\oplus i}_j\wedge u$$ in $V(s\oplus i)$. If $x_i\ne
0\in V(s\oplus i)$ then
 $$\partial_{s\oplus i}^s(r^s_j\wedge
u)=x_i\wedge r^s_j\wedge u=x_i\wedge r^{s\oplus i}_j\wedge
u+\alpha_j x_i\wedge x_i\wedge u= \pm r^{s\oplus i}_j \wedge
(x_i\wedge u).$$ In any case the map $\partial_{s\oplus i}^s$ is
well defined.
\end{proof}

\begin{rema}
For any $s$ state and an index $i$ there are two possibilities for
the map $\partial^s_{s\oplus i}:
\bigwedge^*V(s)\to\bigwedge^*V(s\oplus i)$:
\begin{enumerate}
\item $\rank V(s\oplus i)=\rank V(s)-1$. Then $\partial^s_{s\oplus i}$ is an epimorphism with the
kernel $x_i\bigwedge^*V(s)$;
\item $\rank V(s\oplus i)=\rank V(s)+1$. Then $\partial^s_{s\oplus
i}$ is an isomorphism of $\bigwedge^*V(s)$ onto $x_i
\bigwedge^*V(s\oplus i)$.
\end{enumerate}
\end{rema}


\noindent Every 2-face of the hypercube of states looks like
$$
\xymatrix{ \bigwedge^*V(s\oplus j)\ar[r]^{\partial^{s\oplus j}_{s\oplus i\oplus j}} & \bigwedge^*V(s\oplus i\oplus j)\\
\bigwedge^*V(s)\ar[r]_{\partial^{s}_{s\oplus
i}}\ar[u]^{\partial^{s}_{s\oplus j}} & \bigwedge^*V(s\oplus
i).\ar[u]_{\partial^{s\oplus i}_{s\oplus i\oplus j}} }
$$

According to dimensions of spaces $V(s'),\ s'=s, s\oplus i, s\oplus
j, s\oplus i\oplus j$ we have five types of diagrams:

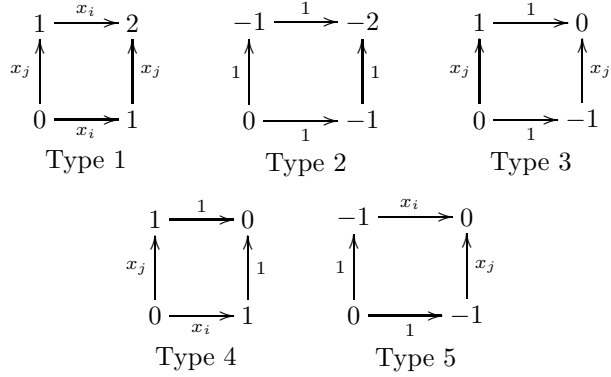
\begin{gather*}
\tar{\xymatrix{ 1\ar[r]^{x_i} & 2\\
0\ar[r]_{x_i}\ar[u]^{x_j} & 1\ar[u]_{x_j}}}{Type 1}\quad
\tar{\xymatrix{ -1\ar[r]^{1} & -2\\
0\ar[r]_{1}\ar[u]^{1} & -1\ar[u]_{1}}}{Type 2}\quad
\tar{\xymatrix{ 1\ar[r]^{1} & 0\\
0\ar[r]_{1}\ar[u]^{x_j} & -1\ar[u]_{x_j}}}{Type 3}\\
\tar{\xymatrix{ 1\ar[r]^{1} & 0\\
0\ar[r]_{x_i}\ar[u]^{x_j} & 1\ar[u]_{1}}}{Type 4}\quad
\tar{\xymatrix{ -1\ar[r]^{x_i} & 0\\
0\ar[r]_{1}\ar[u]^{1} & -1\ar[u]_{x_j}}}{Type 5}
\end{gather*}

Here the number at the place of state $s'$ is equal to $\rank
V(s')-\rank V(s)=\cor(A(s'))-\cor(A(s))$ and the label $z=1, x_i,
x_j$ at the arrow for the map $\partial^{s'}_{s''}$ means that
$\partial^{s'}_{s''}(u) =z\wedge u$.

\begin{propo}[Commutativity of 2-faces]\label{prop_commun_2face}
Any 2-face of the hypercube of states is commutative or
anticommutative.
\end{propo}

\begin{proof}
2-faces of type 1 are anticommutative. 2-faces of types 2,3 are
commutative. A 2-face of type 4 is commutative because $x_i=x_j=0\in
V(s\oplus i\oplus j)$.

We need to look at type 5 more attentively. There are several
possibilities.

1. $\sgn(v_i)=\sgn(v_j)=-1$. Then $v_i, v_j\in s\oplus i\oplus j$.
We can assume that $v_i$ and $v_j$ are the last vertices in $s\oplus
i\oplus j$.

1.1. $v_i,v_j\in\V_0$. The relation matrix $\ol A(s\oplus i\oplus
j)$ can be represented in the form
\begin{equation}\label{Bab1}
\bega{cccc}
& & & \ol B \\
&0 & & \ol a \\
& & & \ol b\\
\ol {-B^\top} &\ol {-a^\top} &\ol{-b^\top}&  0\ena
\end{equation}
where $B=B(s)$. We have $x_i=0\in V(s\oplus i)$. Then the rows of
the matrix $(\ol{-B^\top}\ \ol{-a^\top})$ generate the vector $(0\
1)$. Hence the rows of the matrix $(\ol{-B^\top}\ \ol{-a^\top}\
\ol{-b^\top})$ generate the vector $(0\ 1\ \alpha),\ \alpha\in\Z$.
Analogously, the matrix $(\ol{-B^\top}\ \ol{-a^\top}\ \ol{-b^\top})$
generate the vector $(0\ \beta\ 1)$. These two vectors generate the
vector $(0\ 0\ 1-\alpha\beta)$. If $\alpha\beta\ne 1$ then
$(1-\alpha\beta)x_j=0\in V(s\oplus i\oplus j)$. Since $V(s\oplus
i\oplus j)$ is free we would have $x_j=0\in V(s\oplus i\oplus j)$
but this is not the case. Thus, $\alpha\beta=1$ and
$\alpha=\beta=\pm 1$. Then $x_i\pm x_j=0\in V(s\oplus i\oplus j)$ so
the square is commutative or anticommutative.

1.2. $v_i\in\V_0, v_j\in\V_1$. The relation matrix $\ol A(s\oplus
i\oplus j)$ can be represented in the form
\begin{equation}\label{Bab2}
\bega{cccc}
0&0 & \ol B& \ol{b^\top} \\
0&0 & \ol a & \ol\alpha \\
\ol {-B^\top}& \ol{-a^\top}& 0& 0\\
 \ol{-b}& \ol{-\alpha} & 0&  0\ena.
\end{equation}

Since $x_i=0\in V(s\oplus j)$ the vector $(\ol a\ \ol\alpha)$
depends
on the rows of the matrix $\rank \bega{cc} \ol B &\ol{b^\top}\\
\ol a& \ol\alpha\ena$. Then the vector $\ol a$ depends on the rows
of $\ol B$ so $x_i=0\in V(s)$ but this is not true. Thus, this case
is impossible.

1.3.  $v_i,v_j\in\V_1$. This case can be considered analogously to
the case 1.1.

2. $\sgn(v_i)=-1,\ \sgn(v_j)=1$. Then $v_i\in s\oplus i\oplus j$ and
$v_j\not\in s\oplus i\oplus j$.

2.1. $v_i,v_j\in\V_0$. Without loss of generality we can assume that
the relation matrix $\ol A(s\oplus i)$ has the form~(\ref{Bab1})
where $B=B(s\oplus j)$. Since $x_i=0\in V(s\oplus j)$ then $\rank
\bega{c} \ol B\\ \ol a\ena = \rank \ol B$. Since $x_j=0\in V(s\oplus
j)$ then $\rank \bega{c} \ol B\\ \ol b\ena = \rank \ol B$. Hence,
$\rank \bega{c} \ol B\\ \ol a\\ \ol b\ena = \rank B$ but $\rank
\bega{c} \ol B\\ \ol a\\ \ol b\ena = \rank \bega{c} \ol B\\ \ol
b\ena+1$ because $x_i=0\in V(s\oplus i)$. Thus, this case is
impossible.

2.2.  $v_i\in\V_0, v_j\in\V_1$. The relation matrix $\ol A(s\oplus
i)$ can be represented in the form~\eqref{Bab2} with $B=B(s\oplus
j)$. Since $x_j=0\in V(s\oplus j)$ we have
$\rank\bega{c}\ol{-B^\top}\\ \ol{-b}\ena =\rank\ol{-B^\top}$. Then
the vector $\ol{-b}$ is generated by the rows of $\ol{-B^\top}$ so
the row of the matrix $\bega{cc}\ol{-B^\top} & \ol{-a^\top}\ena$
generate the vector $(\ol{-b} \ \gamma)$. Hence the matrix
$\bega{cc}\ol{-B^\top}& \ol{-a^\top}\\
 \ol{-b}& \ol{-\alpha}\ena$ is equivalent by row transformations to
 the matrix $\bega{cc}\ol{-B^\top}& \ol{-a^\top}\\
 0& \delta\ena$ where $\delta =\ol{-\alpha}-\gamma$. If $\delta=0$
then $\rank \bega{cc}\ol{-B^\top}& \ol{-a^\top}\\
 \ol{-b}& \ol{-\alpha}\ena=\rank \bega{cc}\ol{-B^\top}&
 \ol{-a^\top}\ena$ and $x_j=0\in V(s\oplus i\oplus j)$ but this is
 not true.

 If $|\delta|>1$ then we must have $B=0$ (otherwise we can
 find a minor in $\ol A(G)$ which is not equal to $0$ and is a
 multiple of $\delta$). Hence, $b=0$. If $\ol{-\alpha}=0$ then $x_j=0\in V(\oplus_i\oplus
 j)$ that is not true. If $\ol{-\alpha}=\pm1$ then $x_i\mp x_j=0\in V(\oplus_i\oplus
 j)$.

 If $\delta=\pm 1$ then we have $x_i\mp x_j=0\in V(\oplus_i\oplus
 j)$.

2.3.  $v_i\in\V_1, v_j\in\V_0$. This case is considered analogously
the case 2.2.

2.4. $v_i,v_j\in\V_1$. This case is impossible by the same reason as
the case 2.1.

\noindent 3. $\sgn(v_i)=\sgn(v_j)=-1$. Then $v_i, v_j\in s\oplus
i\oplus j$.

3.1. $v_i,v_j\in\V_0$. Then the matrix $\ol A(s)$ looks
like~(\ref{Bab1}) where $B=B(s\oplus i\oplus j)$. Since $x_i=0\in
V(s\oplus i)$ we have $\rank \bega{c}\ol B\\ \ol a\\ \ol b\ena =
\rank \bega{c}\ol B\\ \ol b\ena$. Then the vector $\ol a$ is
generated by the rows of the matrix $\ol B$ and the vector $\ol b$.
It means that $x_i=\alpha x_j\in V(s\oplus i\oplus j)$. On the other
hand, the same reasoning for $x_j$ leads to the equality $x_i=\beta
x_j\in V(s\oplus i\oplus j)$ so $x_i=\alpha\beta x_i$. If
$\alpha\beta\ne 1$ then $(1-\alpha\beta) x_i=0$ implies $x_i=0$
since $V(s\oplus i\oplus j)$ is free. But $x_i\ne 0\in V(s\oplus
i\oplus j)$. Hence, $\alpha\beta=1$, so $\alpha=\beta=\pm 1$ and
$x_i\mp x_j=0\in V(s\oplus i\oplus j)$.

3.2. $v_i\in\V_0, v_j\in\V_1$. The relation matrix $\ol A(s)$ can be
represented in the form~\eqref{Bab2} with $B=B(s\oplus i\oplus j)$.
Since $x_i\ne 0\in V(s\oplus i\oplus j)$ we have $\rank \bega{c}\ol
B\\ \ol a\ena=\rank \ol B+1$. The equality $x_j\ne 0\in V(s\oplus
i\oplus j)$ implies $\rank \bega{cc}\ol B & \ol {b^\top}\ena=\rank
\ol B+1$. Then $\rank \bega{cc}\ol B & \ol {b^\top}\\ ol a&
\ol\alpha\ena= \rank \bega{c}\ol B\\ \ol a\ena+1$. But $\rank
\bega{cc}\ol B & \ol {b^\top}\\ \ol a& \ol\alpha\ena= \rank
\bega{c}\ol B\\ \ol a\ena$ since $x_j\ne\in V(s)$. Thus, this case
is impossible.

3.3. The case $v_i,v_j\in\V_1$ is analogous to the case 3.1.

Thus, any diagram of type 5 is commutative or anticommutative.
\end{proof}

\begin{rema}
Following~\cite{ORS} we introduce another classification of two
faces: anticommutative faces (type A), commutative faces (type C)
and zero faces (types X and Y). 2-faces of type 1 have type A.
2-faces of types 2,3 have type C. Ass for diagrams of type 5 we
assign to commutative diagrams the type C and to anticommutative
ones the type A. A 2-face of type 4 is a zero face because
$x_i=x_j=0\in V(s\oplus i\oplus j)$, below we assign this face to
type X or Y.
\end{rema}

We call a vertex $v\in V$ {\em inner} if $v\in\V_0$ and $\sgn(v)=-1$
or $v\in\V_1$ and $\sgn(v)=1$. Otherwise $v$ is {\em outer}.

Diagrams of type 4 differ from diagrams of type 5 by a sign of the
vertex $v_i$ or $v_j$. So the consideration of possible cases among
cases 1.1-3.3 in Proposition~\ref{prop_commun_2face} shows this
statement is true.

\begin{lem}\label{lem_commun1}
1. In any diagram of type 5 the vertices $v_i, v_j$ are either both
inner or both outer.\\
2. In any diagram of type 4 one the vertices $v_i, v_j$ is inner and
the other is outer.\hfill$\Box$
\end{lem}

Let us consider a 2-face of type 4 and let $v_i$ be the inner vertex
of the face. We assign the face to the {\em type X} if $x_i=x_j\in
V(s\oplus i)$ and assign to the {\em type Y} if $x_i=-x_j\in
V(s\oplus i)$.

\begin{lem}\label{lem_commun2}
Let us consider a 2-face of type 5. Then \begin{itemize}
\item if $\sgn(v_i)=\sgn(v_j)$ we have $x_i\pm x_j=0\in V(s)$
$\Leftrightarrow$ $x_i\mp x_j=0\in V(s\oplus i\oplus j)$; \item  if
$\sgn(v_i)\ne\sgn(v_j)$ we have $x_i\pm x_j=0\in V(s)$
$\Leftrightarrow$ $x_i\pm x_j=0\in V(s\oplus i\oplus j)$.
\end{itemize}
For a 2-face of type 4 we have \begin{itemize}
\item if $\sgn(v_i)=\sgn(v_j)$ then $x_i\pm x_j=0\in V(s\oplus i)$
$\Leftrightarrow$ $x_i\pm x_j=0\in V(s\oplus j)$;
\item  if $\sgn(v_i)\ne\sgn(v_j)$ then $x_i\pm x_j=0\in V(s\oplus i)$
$\Leftrightarrow$ $x_i\mp x_j=0\in V(s\oplus j)$.
\end{itemize}
\end{lem}
\begin{proof}
Since any diagram of type 4 can be transformed in a diagram of type
5 by change of sign of the vertex $v_i$ we can prove this lemma only
for diagrams of type 5.

Let us denote $s_\alpha=s\cap\V_\alpha,\ \alpha=0,1$.

1.\ Assume at first that $\sgn(v_i)=\sgn(v_j)=-1$ and $v_i, v_j\in
V_0$. Then $v_i,v_j\not\in s$. We have $x_i+\alpha x_j=0\in V(s)$.
This means there exist coefficients $\lambda_k$ where $v_k\in s_0$
such that $$x_i+\alpha x_j=r^s_i+\alpha r^s_j+\sum_{k: v_k\in
s_0}\lambda_k r^s_k\in \Z\langle x_l\ |\ v_l\in
s_1\rangle\oplus\Z\langle x_i,x_j\rangle.$$ In other words, we have
equations $\sum_k\lambda_k
\sgn(v_l)a_{kl}+\sgn(v_l)a_{il}+\alpha\sgn(v_l)a_{jl}=0$ for any
$l\in s_1$. Then $\sum_k\lambda_k a_{kl}=-a_{il}-\alpha\cdot
a_{jl}$.

The equality $x_i+\beta x_j\in V(s\oplus i\oplus j)$ means that
there exist $\mu_l,\ v_l\in  (s\oplus i\oplus j)_1=s_1$ such that
$$
x_i+\beta x_j=\sum_{l\ :\ v_l\in (s\oplus i\oplus j)_1}\mu_l
r^{s\oplus i\oplus j}_l\in \Z\langle x_k\ |\ v_k\in (s\oplus i\oplus
j)_0\rangle.
$$ This
is equivalent to the system of equations:
$$\sum_l \mu_l a_{lk}=0,\
v_k\in s_0,\quad \sum_l \mu_l a_{li}=-\sgn(v_i),\quad \sum_l \mu_l
a_{lj}=-\beta\sgn(v_j).
$$

Then \begin{multline*} \sum_{(k,l)\ :\ v_k\in s_0,\ v_l\in
s_1}\lambda_k\mu_l a_{kl}=\sum_l\mu_l\sum_k\lambda_k
a_{kl}=-\sum_l\mu_l (a_{il}+\alpha\cdot a_{jl})=\\
\sum_l\mu_l
a_{li}+\alpha\sum_l\mu_la_{lj}=-\sgn(v_i)-\alpha\beta\sgn(v_j).
\end{multline*}
On the other hand, $\sum\limits_{k,l}\lambda_k\mu_l
a_{kl}=-\sum_k\lambda_k\sum_l\mu_la_{lk}=0.$ Thus,
$\sgn(v_i)\sgn(v_j)+\alpha\beta=0$ so $\alpha=-\beta$ that proves
the statement of the lemma.

The case $v_i, v_j\in V_1$ is considered analogously.

2.\ Assume that $\sgn(v_i)=-1,\ \sgn(v_j)=1$ and $v_i\in V_0,\
v_j\in V_1$. We denote $s'_\alpha= (s\oplus j)\cap \V_\alpha,\
\alpha=0,1$. Then $s_0=s'_0,\ s_1=s'_1\cup\{v_j\}$ and $(s\oplus
i\oplus j)_0=s'_0\cup\{v_i\},\ (s\oplus i\oplus j)_1=s'_1$.

The identity $x_i+\alpha x_j=0\in V(s)$ means there exist
$\lambda_k,\ v_k\in s_0,$ such that
$$ x_i+\alpha x_j = r^s_i+\sum_{k\ :\ v_k\in s'_0}\lambda_k r^s_k \in \Z\langle x_k\ |\ v_k\in
s_1\rangle\oplus\Z\langle x_i\rangle.
$$
Then we have equations
$$\sum_k\lambda_k a_{kl}+a_{il}=0,\ v_l\in
s'_1,\quad \sum_k\lambda_k a_{kj}+a_{ij}=-\alpha\sgn(v_j).$$

The identity $x_i+\beta x_j=0\in V(s\oplus i\oplus j)$ leads to
equations
$$\sum\limits_{l\ :\ v_l\in s'_1}\mu_l a_{lk}=\beta
a_{jk},\ v_k\in s'_0,\quad \sum\limits_{l\ :\ v_l\in s'_1}\mu_l
a_{li}+\beta a_{ji}=-\sgn(v_i).$$

Then
$$\sum\limits_{(k,l)\ :\ v_k\in s'_0,\ v_l\in
s'_1}\lambda_k\mu_l a_{kl}=\sum_l\mu_l\sum_k\lambda_k
a_{kl}=-\sum_l\mu_l a_{il}=\sgn(v_i)+\beta a_{ji}.
$$
On the other hand,
\begin{multline*} \sum_{(k,l)\ :\ v_k\in s'_0,\ v_l\in
s'_1}\lambda_k\mu_l a_{kl}=-\sum_k\lambda_k\sum_l\mu_l
a_{lk}=\beta\sum_k\lambda_l a_{kj}=-\beta a_{ij}-\alpha\beta
\sgn(v_j).
\end{multline*}
Since $a_{ji}=-a_{ij}$ we have $\alpha\beta=-\sgn(v_i)\sgn(v_j)=1$.
Thus $x_i\pm x_j=0\in V(s)$ iff $x_i\pm x_j=0\in V(s\oplus i\oplus
j)$.

\noindent 3.\ The case $\sgn(v_i)=\sgn(v_j)=1$ is considered
analogously the case 1.
\end{proof}

Let us consider a 2-face of type 4. We assign the face to the {\em
type X} if $x_i=\sgn(v_j)x_j\in V(s\oplus i)$ (by
Lemma~\ref{lem_commun2} this is equivalent to $x_i=\sgn(v_i)x_j\in
V(s\oplus j)$) and assign to the {\em type Y} if
$x_i=-\sgn(v_j)x_j\in V(s\oplus i)$.

\medskip

\noindent{\em Edge assignment.}

Let us denote the set of the edges in the hypercube as $\E$. We call
{\em edge assignment} any map $\eps : \E\to
\{\pm1\}$(see~\cite{ORS}). A 2-face is called {\em even} (resp.{\em
odd}) if it contains even (resp. odd) number of edges $e$ with
$\eps(e)=-1$. A {\em type X edge assignment} is an edge assignment
such that all faces of type A and X are even and all faces of type C
and Y are odd. Similarly,  {\em type Y edge assignment} is an edge
assignment for which faces of type A and Y are even and faces of
type C and X are odd.

\begin{lem}
Each cube in the hypercube contains an even number of squares of
type A and X. Similarly, each cube contains an even number of
squares of type A and Y.
\end{lem}
\begin{proof}
The proof is an analysis of all possible configurations of cubes
using Lemmas~\ref{lem_commun1},\ref{lem_commun2}.
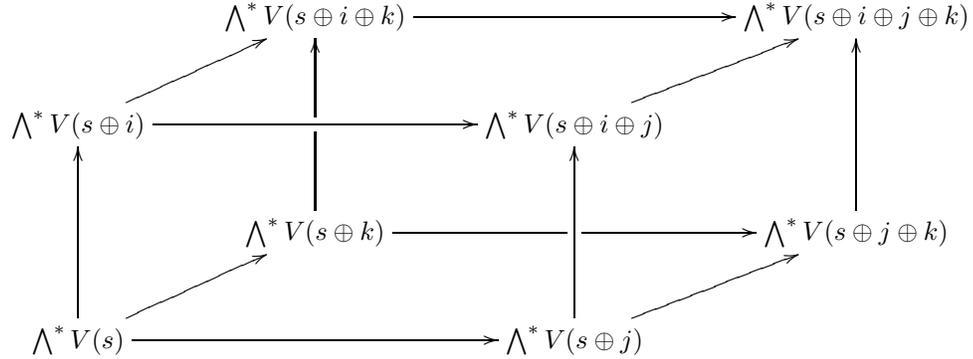
$$
\xymatrix{ & \bigwedge^*V(s\oplus i\oplus k) \ar[rr] & & \bigwedge^*V(s\oplus i\oplus j\oplus k) \\
\bigwedge^*V(s\oplus i) \ar[ur]\ar[rr] & & \bigwedge^*V(s\oplus i\oplus j)\ar[ur] &\\
& \bigwedge^*V(s\oplus k)\ar'[r][rr]\ar'[u][uu] && \bigwedge^*V(s\oplus j\oplus k)\ar[uu]\\
\bigwedge^*V(s)\ar[rr]\ar[ur]\ar[uu] && \bigwedge^*V(s\oplus
j)\ar[ur]\ar[uu] &}
$$
There are 18 possible cubes (up to symmetry of axes). Below the
number at the place of state $s'$ in the cube is equal to $\rank
V(s')-\rank V(s)$. These cases can be classified into the following
groups.

$$
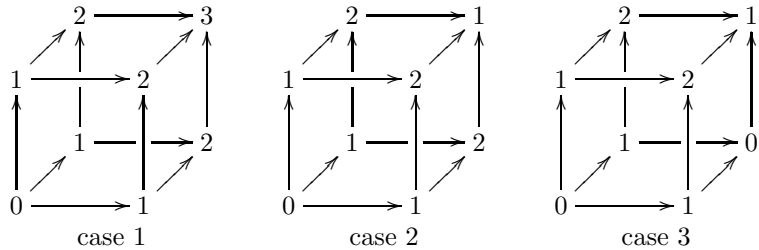
\tar{\xymatrix@!0{ & 2 \ar[rr] & & 3 \\
1 \ar[ur]\ar[rr] & & 2 \ar[ur] &\\
& 1\ar'[r][rr]\ar'[u][uu] && 2\ar[uu]\\
0\ar[rr]\ar[ur]\ar[uu] && 1\ar[ur]\ar[uu] &}}{case 1}\quad
\tar{\xymatrix@!0{ & 2 \ar[rr] & & 1 \\
1 \ar[ur]\ar[rr] & & 2 \ar[ur] &\\
& 1\ar'[r][rr]\ar'[u][uu] && 2\ar[uu]\\
0\ar[rr]\ar[ur]\ar[uu] && 1\ar[ur]\ar[uu] &}}{case 2}\quad
\tar{\xymatrix@!0{ & 2 \ar[rr] & & 1 \\
1 \ar[ur]\ar[rr] & & 2 \ar[ur] &\\
& 1\ar'[r][rr]\ar'[u][uu] && 0\ar[uu]\\
0\ar[rr]\ar[ur]\ar[uu] && 1\ar[ur]\ar[uu] &}}{case 3}
$$

$$
\tar{\xymatrix@!R0{ & 0 \ar[rr] & & 1 \\
1 \ar[ur]\ar[rr] & & 0 \ar[ur] &\\
& 1\ar'[r][rr]\ar'[u][uu] && 2\ar[uu]\\
0\ar[rr]\ar[ur]\ar[uu] && 1\ar[ur]\ar[uu] &}}{case 4}\quad
\tar{\xymatrix@!0{ & 0 \ar[rr] & & 1 \\
1 \ar[ur]\ar[rr] & & 0 \ar[ur] &\\
& 1\ar'[r][rr]\ar'[u][uu] && 0\ar[uu]\\
0\ar[rr]\ar[ur]\ar[uu] && 1\ar[ur]\ar[uu] &}}{case 5}\quad
\tar{\xymatrix@!0{ & 0 \ar[rr] & & 1 \\
-1 \ar[ur]\ar[rr] & & 0 \ar[ur] &\\
& 1\ar'[r][rr]\ar'[u][uu] && 2\ar[uu]\\
0\ar[rr]\ar[ur]\ar[uu] && 1\ar[ur]\ar[uu] &}}{case 6}
$$

$$
\tar{\xymatrix@!R0{ & 0 \ar[rr] & & 1 \\
-1 \ar[ur]\ar[rr] & & 0 \ar[ur] &\\
& 1\ar'[r][rr]\ar'[u][uu] && 0\ar[uu]\\
0\ar[rr]\ar[ur]\ar[uu] && 1\ar[ur]\ar[uu] &}}{case 7}\quad
\tar{\xymatrix@!0{ & 0 \ar[rr] & & 1 \\
1 \ar[ur]\ar[rr] & & 0 \ar[ur] &\\
& -1\ar'[r][rr]\ar'[u][uu] && 0\ar[uu]\\
0\ar[rr]\ar[ur]\ar[uu] && -1\ar[ur]\ar[uu] &}}{case 8}\quad
\tar{\xymatrix@!0{ & 0 \ar[rr] & & 1 \\
-1 \ar[ur]\ar[rr] & & 0 \ar[ur] &\\
& -1\ar'[r][rr]\ar'[u][uu] && 0\ar[uu]\\
0\ar[rr]\ar[ur]\ar[uu] && -1\ar[ur]\ar[uu] &}}{case 9}
$$

$$
\tar{\xymatrix@!R0{ & 0 \ar[rr] & & -1 \\
1 \ar[ur]\ar[rr] & & 0 \ar[ur] &\\
& 1\ar'[r][rr]\ar'[u][uu] && 0\ar[uu]\\
0\ar[rr]\ar[ur]\ar[uu] && 1\ar[ur]\ar[uu] &}}{case 10}\quad
\tar{\xymatrix@!0{ & 0 \ar[rr] & & -1 \\
-1 \ar[ur]\ar[rr] & & 0 \ar[ur] &\\
& 1\ar'[r][rr]\ar'[u][uu] && 0\ar[uu]\\
0\ar[rr]\ar[ur]\ar[uu] && 1\ar[ur]\ar[uu] &}}{case 11}\quad
\tar{\xymatrix@!0{ & 0 \ar[rr] & & -1 \\
1 \ar[ur]\ar[rr] & & 0 \ar[ur] &\\
& -1\ar'[r][rr]\ar'[u][uu] && 0\ar[uu]\\
0\ar[rr]\ar[ur]\ar[uu] && -1\ar[ur]\ar[uu] &}}{case 12}
$$

$$
\tar{\xymatrix@!R0{ & 0 \ar[rr] & & -1 \\
1 \ar[ur]\ar[rr] & & 0 \ar[ur] &\\
& -1\ar'[r][rr]\ar'[u][uu] && -2\ar[uu]\\
0\ar[rr]\ar[ur]\ar[uu] && -1\ar[ur]\ar[uu] &}}{case 13}\quad
\tar{\xymatrix@!0{ & 0 \ar[rr] & & -1 \\
-1 \ar[ur]\ar[rr] & & 0 \ar[ur] &\\
& -1\ar'[r][rr]\ar'[u][uu] && 0\ar[uu]\\
0\ar[rr]\ar[ur]\ar[uu] && -1\ar[ur]\ar[uu] &}}{case 14}\quad
\tar{\xymatrix@!0{ & 0 \ar[rr] & & -1 \\
-1 \ar[ur]\ar[rr] & & 0 \ar[ur] &\\
& -1\ar'[r][rr]\ar'[u][uu] && -2\ar[uu]\\
0\ar[rr]\ar[ur]\ar[uu] && -1\ar[ur]\ar[uu] &}}{case 15}
$$

$$
\tar{\xymatrix@!R0{ & -2 \ar[rr] & & -1 \\
-1 \ar[ur]\ar[rr] & & -2 \ar[ur] &\\
& -1\ar'[r][rr]\ar'[u][uu] && 0\ar[uu]\\
0\ar[rr]\ar[ur]\ar[uu] && -1\ar[ur]\ar[uu] &}}{case 16}\quad
\tar{\xymatrix@!0{ & -2 \ar[rr] & & -1 \\
-1 \ar[ur]\ar[rr] & & -2 \ar[ur] &\\
& -1\ar'[r][rr]\ar'[u][uu] && -2\ar[uu]\\
0\ar[rr]\ar[ur]\ar[uu] && -1\ar[ur]\ar[uu] &}}{case 17}\quad
\tar{\xymatrix@!0{ & -2 \ar[rr] & & -3 \\
-1 \ar[ur]\ar[rr] & & -2 \ar[ur] &\\
& -1\ar'[r][rr]\ar'[u][uu] && -2\ar[uu]\\
0\ar[rr]\ar[ur]\ar[uu] && -1\ar[ur]\ar[uu] &}}{case 18}
$$

Cases 1,6,13,18. The types of 2-faces are determined by ranks of
states and their number can be counted explicitly. For example, in
the case 6 the cube contains 4 faces of type C and 2 faces of type
A.

Cases 2,5,7,10,14,15. These cases are not realizable because of
Lemma~\ref{lem_commun1}. One can not attribute the vertices $v_i,
v_j, v_k$ to inner and outer vertices so that the faces of type 4,
incident to the state $s$ or $s\oplus i\oplus j\oplus k$, have right
configuration.

Cases 3,8,11,16. The faces of the cube include 4 commutative faces
(cases 11, 16) or 2 A-faces and 2 C-faces (cases 3,8). The other two
faces are opposite to each other and have the same type because
there is a projection from one face to the other.

Cases 4,9,12,17. For example, let us consider the case 4. The cube
has one anticommutative face and 2 commutative faces. The other
three faces are incident to the state $s\oplus i$. Let us assume
that $x_i=x_j=x_k\in V(s\oplus i)$ and
$\sgn(v_i)=\sgn(v_j)=\sgn(v_k)=-1$. Then these three faces have
types A, Y and Y. We can see that if one changes the sign of any
vertex $v_i, v_j, v_k$ or the sign of any variable $x_i, x_j, x_k$
in $V(s\oplus i)$ then two faces of the three change their type. So
the number of A- and X-faces remains even.
\end{proof}

\begin{lem}
Any PU-oriented bipartite labeled graph $G$ has an edge assignment
of type X and one of type Y.
\end{lem}
\begin{proof}
See~\cite[Lemma 1.2]{ORS}.
\end{proof}

Given a type X or type Y edge assignment $\eps$ we define the chain
complex
$$C(G)=\bigoplus\limits_{s\subset\V}\bigwedge\nolimits^*V(s)$$
with differential
$$\partial_\eps(u)=\sum\limits_{\{s,s'\subset\V\,|\, s\rightarrow s'=\eps\in\E\}}
\eps(e)\partial_{s'}^s(u).$$

This lemmas were proved in~\cite{ORS}.
\begin{lem}
If $\eps$ and $\eps'$ are two edge assignment of the same type (X or
Y) then the chain complex $(C(G),\partial_\eps)$ is isomorphic to
$(C(G),\partial_\eps')$.\hfill$\Box$
\end{lem}

\begin{lem}
If $\eps$ and $\eps$ are two edge assignment of opposite types then
there is isomorphism
$(C(G),\partial_\eps)\cong(C(G),\partial_\eps')$.\hfill$\Box$
\end{lem}



\begin{defi}
Homology $\Kh(G)$ of the complex $(C(G),\partial)$ is called {\em
reduced odd Khovanov homology} of the labeled simple graph $G$.
\end{defi}

The main theorem of the article states that odd Khovanov homology is
in fact an invariant of PU-orientable graph-links.

\begin{theor}
Khovanov homology $\Kh(G)$ is invariant under $R, \Omega_1, \ofo,
\Omega_3, \Omega_4$ moves.
\end{theor}

\begin{proof}
Let $G$ be a labeled graph and $\wt G$ be a graph obtained from $G$
by some Reidemeister move $R, \Omega_1, \ofo, \Omega_3, \Omega_4$.

\medskip
\noindent{\em Invariance under $R$.}

See~\cite[Lemma 2.3]{ORS}.

\medskip
\noindent{\em Invariance under $\Omega_1$.}

Let $\widetilde G$ be obtained from $G$ by addition an isolated
labeled vertex $v$. Then its adjacency matrix $A(\wt G)$ looks like
$\bega{cc} 0 & 0\\ 0& A(G)\ena$. The complex $C(\wt G)$ splits as a
$\Z$-module into the sum $C\oplus C_v$ where $C$ corresponds to
states $s\in\V(\wt G)$ such that $v\not\in s$ and $C_v$ corresponds
to states $s\in\V(\wt G)$ such that $v\in s$. There is a natural
bijecticton between states of $C(G)$, $C$ and $C_v$. Let $\eps$ be
an edge assignment on $C(G)$. We define an edge assignment $\wt\eps$
on $C(\wt G)$ as follows.

We set $\wt\eps=\eps$ on $C_v$ and $\wt\eps(e)=1$ for all edges
between $C$ and $C_v$. If $\sgn(v)=1$ we set $\wt\eps=-\eps$ on $C$.
If $\sgn(v)=-1$ we set $\wt\eps=\delta\cdot\eps$ on $C$ where
$\delta(s\to s\oplus i)=1$ if $\rank V(s\oplus i)=\rank V(s)+1$ and
$\delta(s\to s\oplus i)=-1$ otherwise.

Then $\wt\eps$ is an edge assignment of the same type as $\eps$
because all squares with edges connecting $C$ and $C_v$ are of type
C or A and the parity of other squares is the same as in $C(G)$.

The complex $(C(\wt G),\partial_{\wt\eps})$ is isomorphic to product
of complexes $(C(G),\partial_\eps)\otimes C(v)$ where the complex
$C(v)$ is equal to
$$\xymatrix{\Z_2 \ar[r]^-{x\wedge} & \bigwedge^*\Z_2\langle
x\rangle}$$ if $\sgn(v)=-1$ and
$$\xymatrix{\bigwedge^*\Z_2\langle x\rangle \ar[r]^-{x=0} & \Z_2}$$ if
$\sgn(v)=1$. In any case $H_*(C(v))=\Z_2\cdot 1$, where $1\in
H_0(C(v))$ if $\sgn(v)=1$ and $1\in H_1(C(v))$ if $\sgn(v)=-1$.
Thus, we have
$$ \Kh(\wt G) = \Kh(G)\otimes \Kh(v)\cong\Kh(G).$$

\noindent{\em Invariance under $\Omega_2$.}

Assume that we add vertices $v$ and $w$ to get the graph $\wt G$ by
$\Omega_2$ and $\sgn(v)=1$, $\sgn(w)=-1$ . We can write the
adjacency matrix $A(\wt G)$ in the form
$$
\bega{ccc}  0&0& a \\
0&0& a\\
-a^\top&-a^\top&A(G)\ena.
$$
For every state $s\in \V(G)$ the following equations $\cor
A(G(s))=\cor A(\wt G(s))$, $\cor A(G(s\cup\{v\}))=\cor
A(G(s\cup\{w\}))=\cor A(G(s\cup\{v,w\}))-1$.  These equalities
defines the type of the upper and left arrows of the complex $C(\wt
G)$ written in the form
$$
\xymatrix{C_{vw}\ar[r]^{1} & C_w\\
C_v\ar[u]^{x_2}\ar[r]_\partial & C.\ar[u]_\partial}
$$
Here $C_v$ is consists of chains whose state contains $v$ and does
not contain $w$, $C, C_w, C_{vw}$ are defined analogously. Notice
that the constriction of the edge assignment $\wt \eps$ of the
complex $C(\wt G)$ to $C$ is an edge assignment of the same type as
$\wt\eps$.

For any state $s$ in $C_{vw}$ define linear function $f : V(s)\to
\Z_2$ by formula $f(\sum_i\lambda_ix_i)=\lambda_1+\lambda_2$.
Function $f$ is well-defined because it vanishes any relation:
$f(r^s_i)=\ol a_{i1}+\ol a_{i2}=0$ since $a_{i1}=a_{i2}$ and
$\sgn(v_i)=-\sgn(v_j)$. Then $\bigwedge^*V(s)=\bigwedge^*\ker
f\oplus x_2\bigwedge^*V(s)$ and $C_{vw}=X\oplus x_2C_{vw}$. One can
check that subcomplex $X\rightarrow C_{w}$ is acyclic. Then homology
of $C(\wt G)$ coincides with the homology of the quotient complex
$$
\xymatrix{x_2C_{vw} & \\
C_v\ar[u]^{x_2}\ar[r]_\partial & C.}
$$
The quotient of this complex by subcomplex $C$ appears to be acyclic
too. Thus $C(\wt G)$ has the same homology as $C=C(G)$.

\medskip
\noindent{\em Invariance under $\Omega_3$.}

Let the vertices $u, v, w$ in the third Reidemeister move has the
numbers $1, 2, 3$ in the $\V(G)=\V(\wt G)$. Then the adjacency
matrices of $G$ and $\wt G$ looks like
$$
A(G)=\bega{cccc} 0&1&1&0\\
-1&0&0&a\\
-1&0&0&b\\
0^\top&-a^\top&-b^\top&D\ena,\quad
A(\wt G)=\bega{cccc} 0&0&0&a-b\\
0&0&0&a\\
0&0&0&b\\
(b-a)^\top&-a^\top&-b^\top&D\ena.
$$
The correspondent relation matrices are
$$
\ol A(G)=\bega{cccc} 0&1&1&0\\
-1&0&0&\ol a\\
-1&0&0&\ol b\\
0^\top&-a^\top&-b^\top&\ol D\ena,\quad
\ol A(\wt G)=\bega{cccc} 0&0&0&\ol{a-b}\\
0&0&0&\ol a\\
0&0&0&\ol b\\
(b-a)^\top&a^\top&b^\top&\ol D\ena.
$$

Denote $\wt V(s)=V(\wt G(s))$. Then for any $s\subset
\V(G)\setminus\{u,v,w\}$ we have $V(s)\cong \wt V(s)$, $V(s\oplus
v)\cong \wt V(s\oplus v)$, $V(s\oplus w)\cong \wt V(s\oplus w)$,
$V(s\oplus v\oplus w)\cong \wt V(s\oplus u\oplus v)\cong \wt
V(s\oplus u\oplus w)$, $V(s\oplus u\oplus v\oplus w)\cong \wt
V(s\oplus u)$, $V(s\oplus u\oplus v)\cong V(s\oplus u\oplus w)\cong
V(s)$ and the correspondent isomorphisms of the exterior algebras
compatible with the differential.

Consider complexes $C(G)$ and $C(\wt G)$ in the form of cube:
$$
\xymatrix@!0{ & C_{uw} \ar[rr] & & C_{uvw} \\
C_u \ar[ur]^1\ar[rr]^(0.65)1 & & C_{uv}\ar[ur] &\\
& C_w\ar'[r][rr]\ar'[u][uu] && C_{vw}\ar[uu]\\
C\ar[rr]\ar[ur]\ar[uu]^{x_1} && C_v\ar[ur]\ar[uu] &} \qquad
\xymatrix@!0{ & \wt C_{uv} \ar[rr] & & \wt C_u \\
\wt C_{uvw} \ar[ur]^1\ar[rr]^(0.65)1 & & \wt C_{uw}\ar[ur] &\\
& \wt C_v\ar'[r][rr]\ar'[u][uu] && \wt C\ar[uu]\\
\wt C_{vw}\ar[rr]\ar[ur]\ar[uu]^{x_1} && \wt C_w\ar[ur]\ar[uu] &}
$$

For any state $s$ in $C_{u}$ define a linear function $f : V(s)\to
\Z_2$ by the formula $f(\sum_i\lambda_ix_i)=\lambda_1$. The function
is well-defined and there are decompositions
$\bigwedge^*V(s)=\bigwedge^*\ker f\oplus x_1\bigwedge^*V(s)$ and
$C_{u}=X\oplus x_2C_{u}$. Consider the following subcomplex
$$
\xymatrix@!0{ & C_{uw} \ar[rr] & & C_{uvw} \\
X \ar[ur]^1\ar[rr]^(0.65)1 & & C_{uv}\ar[ur] &\\
& C_w\ar'[r][rr]\ar'[u][uu] && C_{vw}\ar[uu]\\
 && C_v\ar[ur]\ar[uu] &}
$$
Factor complex $C\to x_2C_{u}$ is acyclic so homology of $C(G)$ is
isomorphic to homology of the subcomplex. This subcomplex contains
acyclic subcomplex $X\to \partial(X)$. The maps $X\to C_{uv}$ and
$X\to C_{uw}$ are isomorphisms, so after factorization we get the
complex
$$
\xymatrix@!0{ & C_{uw} \ar[rr]\ar@{=}[dr] & & C_{uvw} \\
& & C_{uv}\ar[ur] &\\
& C_w\ar'[r][rr]\ar[uu] && C_{vw}\ar[uu]\\
 && C_v\ar[ur]\ar[uu] &}
$$

Analogous reasonings reduce $C(\wt G)$ to the complex (in $\wt
C_{uvw}$ we should define the function $f : V(s)\to \Z_2$ by the
formula $f(\sum_i\lambda_ix_i)=\lambda_1+\lambda_2-\lambda_3$)
$$
\xymatrix@!0{ & \wt C_{uv} \ar[rr]\ar@{=}[dr] & & \wt C_{u} \\
& & \wt C_{uw}\ar[ur] &\\
& \wt C_v\ar'[r][rr]\ar[uu] && C\ar[uu]\\
 && \wt C_w\ar[ur]\ar[uu] &}
$$
We show now that these complexes are isomorphic.

Let $\eps$ be an edge assignment of type X on $C(G)$ . We denote
$\eps_\emptyset, \eps_u, \eps_{uv} $ etc. the restriction of $\eps$
to the edges of the spaces $C, C_u, C_{uv}$ etc. The restriction of
the edge assignment $\eps$ to the edges between cubes $C_u$ and
$C_{uv}$ will be denoted as $\eps_u(v)$, for edges between other
cubes we use similar notation.

As in~\cite{ORS} we can assume that $\eps_{uv}=\eps{uw}$.

We need explicit expression for the isomorphisms $V(s\oplus v\oplus
w)\cong \wt V(s\oplus u\oplus v)\cong \wt V(s\oplus u\oplus w)$ etc.
mentioned above. We define these isomorphisms on the generators as
shown in the table

\begin{equation}\label{isom_V}
\begin{array}{|c|c|c|c|}
\hline
& x_1 & x_2 & x_3\\
\hline
V_{uv}\to V & -x_2 & 0 & x_3-x_2\\
V_{uw}\to V & -x_3 & x_2-x_3 & 0\\
V_{uvw}\to \wt V_u & -x_2=-x_3 & x_1 & -x_1\\
V_{vw}\to \wt V_{uv} & x_2 & x_1-x_2 & -x_1\\
V_{vw}\to \wt V_{uw} & x_3 & x_1 & -x_1-x_3\\
\hline
\end{array}
\end{equation}
Any other generator $x_i,\ i\ge 4,$ goes to itself: $x_i\mapsto
x_i$.

These maps induce isomorphisms $\phi^{uv} : C_{uv}\to C$, $\phi^{uw}
: C_{uw}\to C$, $\phi^{uvw} : C_{uvw}\to \wt C_u$, $\phi^{vw}_{uv} :
C_{vw}\to \wt C_{uv}$, $\phi^{vw}_{uw} : C_{vw}\simeq \wt C_{uw}$.
Below for identification the cube $C_{uw}$ with $C$ we will use the
modified isomorphism $\phi^{uw}_+=\delta\phi^{uw}$ where
$\delta=-\eps_u(v)\eps_u(w)$. Then we choose the edge assignment in
the image to extend the isomorphisms of spaces to chain maps of
squares
\begin{gather*}
\vcenter{\xymatrix@!{C_{uv}\ar[r]^{\eps_{uv}(w)} & C_{uvw}\\
C_{v}\ar[u]^{\eps_v(u)}\ar[r]_{\eps_v(w)} &
C_{vw}\ar[u]_{\eps_{vw}(u)}}}
\quad\longrightarrow\quad \vcenter{\xymatrix@!{C\ar[r]^{\wt\lambda_1\eps_{uv}(w)} &\wt C_{u}\\
C_{v}\ar[u]^{\wt\lambda_2\eps_v(u)}\ar[r]_{\wt\lambda_1\eps_v(w)} &\wt C_{uv}\ar[u]_{\wt\lambda_2\eps_{vw}(u)}}} \\
\vcenter{\xymatrix@!{C_{uw}\ar[r]^{\eps_{uw}(v)} & C_{uvw}\\
C_{w}\ar[u]^{\eps_w(u)}\ar[r]_{\eps_w(v)} &
C_{vw}\ar[u]_{\eps_{vw}(u)}}} \quad\longrightarrow\quad
\vcenter{\xymatrix@!{C\ar[r]^{\delta\eps_{uw}(v)} &\wt C_{u}\\
C_{w}\ar[u]^{\wt\lambda_3\delta\eps_w(u)}\ar[r]_{\eps_w(v)} &\wt
C_{uw}\ar[u]_{\wt\lambda_3\eps_{vw}(u)}}}
\end{gather*}
The induced edge assignment $\wt\eps$ is indicated on the edges of
the diagrams. The sign $\wt\lambda_i,\ i=1,2,3,$ is equal to $-1$ on
edges $s\to s\oplus i$ in $C(\wt G)$, such that $\rank V(s\oplus
i)=\rank V(s)+1$, and equal to $1$ otherwise.

The maps $C_{uv}\to C_{uvw}$ and $C_{uw}\to C_{uvw}$ induce the same
maps $C\to \wt C_{u}$. Indeed, the first square gives the map
$\wt\eps_\emptyset(u)\partial^s_{s\oplus
1}=\eps_{uv}(w)\phi^{uvw}\partial^{s\oplus 1\oplus 2}_{s\oplus
1\oplus 2\oplus 3}(\phi^{uv})^{-1}$ whereas the second gives
$\eps_{uw}(v)\phi^{uvw}\partial^{s\oplus 1\oplus 3}_{s\oplus 1\oplus
2\oplus 3}(\phi^{uw}_+)^{-1}$. The identification map
$(\phi^{uw}_+)^{-1}\phi^{uv}$ between spaces $C_{uv}$ and $C_{uw}$
is equal to $\delta \partial^{s\oplus 1}_{s\oplus 1\oplus
3}(\partial^{s\oplus 1}_{s\oplus 1\oplus 2})^{-1}$. Then the two
induced maps coincide if
$$\eps_{uv}(w)\partial^{s\oplus 1\oplus 2}_{s\oplus
1\oplus 2\oplus 3}\partial^{s\oplus 1}_{s\oplus 1\oplus
2}=\delta\eps_{uw}(v)\partial^{s\oplus 1\oplus 3}_{s\oplus 1\oplus
2\oplus 3}\partial^{s\oplus 1}_{s\oplus 1\oplus 3}.
$$
But $\delta=-\eps_u(v)\eps_u(w)$ and
$$\eps_{uv}(w)\eps_u(v)\phi^{uvw}\partial^{s\oplus 1\oplus 2}_{s\oplus
1\oplus 2\oplus 3}\partial^{s\oplus 1}_{s\oplus 1\oplus
2}=-\eps_{uw}(v)\eps_u(w)\phi^{uvw}\partial^{s\oplus 1\oplus
3}_{s\oplus 1\oplus 2\oplus 3}\partial^{s\oplus 1}_{s\oplus 1\oplus
3}$$ because any 2-face in $C(G)$ equipped with an edge assignment
anticommmutes.

The first isomorphism of squares keeps the types of 2-faces and does
not change the parity with respect to the edge assignments. Then the
induced edge assignment has type X on the square
$$
\xymatrix@!{C\ar[r] &\wt C_{u}\\
C_{v}\ar[u]\ar[r] &\wt C_{uv}\ar[u]}
$$

The second isomorphism keeps the parity of 2-faces of types 1,2,3,
change the parity for types 4 and 5 because in this case there exist
a unique edge of the square with $\wt\lambda_1=-1$. On the other
hand, the isomorphism interchange squares A and C of type 5 since
the relation $x_1=x_2\in V_{uvw}$ is equivalent to $-x_3=x_1\in \wt
V_u$. X-squares become Y-squares and vice versa because
$x_2=\sgn(v_1)x_1 (=-x_1)\in V_{vw}$ turns into
$x_1=-x_3=-\sgn(v_3)x_3\in\wt V_{uw}$. Thus, $\wt\eps$ again appears
to be of type X on the square
$$
\xymatrix@!{C\ar[r] &\wt C_{u}\\
C_{w}\ar[u]\ar[r] &\wt C_{uw}\ar[u]}
$$
We can extend $\wt\eps$ to an edge assignment of type X on the whole
complex $C(\wt G)$ (it is possible because the quotient space
obtained after one collapses the faces where $\wt\eps$ is defined to
a point has vanishing cohomology group $H^2$). Then contraction of
the complex $C(\wt G)$ with the chosen edge assignment yields the
complex isomorphic to the corresponding complex obtained from
$C(G)$.

\medskip
\noindent{\em Invariance under $\Omega_4$.}

Let the vertices $u$ and $v$ of the move $\Omega_4$ has numbers $p$
and $q$ in $V(G)=V(\wt G)$. The coefficients of adjacency matrices
of $A(G)=(a_{ij})$ and $A(\wt G)=(\wt a_{ij})$ are connected by the
formula
$$
\wt a_{ij} = \left\{\begin{array}{cl}
a_{ij}-a_{pq}a_{ip}a_{jq}+a_{pq}a_{iq}a_{jp},& \{i,j\}\cap \{p,q\}=\emptyset,\\
a_{ij}, & \{i,j\}\cap \{p,q\}\ne\emptyset,\{p,q\},\\
-a_{ij}, & \{i,j\}= \{p,q\}.
\end{array}\right.
$$

Consider the map $\phi$ acting on the states by the formula
$$
\phi(s)=\left\{\begin{array}{cl} s\cup\{u,v\},& \{u,v\}\cap
s=\emptyset, \\
s\setminus \{u,v\}, & \{u,v\}\cap s=\{u,v\},\\
s, & \{u,v\}\cap s\ne\emptyset, \{u,v\}\\
\end{array}\right.
$$
with the linear maps $\Phi : V(s)\to V(\phi(s))$ defined by the
formula

$$
\Phi(x_i)=\left\{\begin{array}{cl} x_i,& i\ne p,q, \\
x_q, & i=p,\\
x_p, & i=q.\\
\end{array}\right.
$$

Then the map $\Phi$ is well-defined and after natural extension to
homomorphisms of external algebras it determines map $\Phi:C(G)\to
C(\wt G)$. The map $\Phi$ does not change the types of 2-faces. If
we choose the edge assignment $\eps$ and $\wt\eps$ on the spaces
$C(G)$ and $C(\wt G)$ respectively, such that $\eps =
\phi^*(\wt\eps)$ then $\Phi$ appears to be a chain map. Thus the
complexes $(C(G),\partial_\eps)$ and $(C(\wt G),\partial_{\wt\eps})$
are isomorphic as well as their homology.
\end{proof}

\begin{coro}
Odd Khovanov homology $\Kh(G)$ is an invariant of  bipartite
PU-graph-links.\hfill $\Box$
\end{coro}

The author is grateful to V.~Manturov, D.~Ilyutko and J.~Bloom for
inspiring discussions.

\end{document}